\documentclass[11pt]{article}
\usepackage{geometry}                
\usepackage{graphicx}
\usepackage{amssymb}
\usepackage{epstopdf}
\usepackage{amsthm}
\newtheorem{thm}{Theorem}[section]

\newtheorem{lemma}[thm]{Lemma}
\newtheorem{corr}[thm]{Corollary}
\newtheorem{fact}[thm]{Fact}

\usepackage[colorlinks=true, pdfstartview=FitV, linkcolor=blue,
            citecolor=blue, urlcolor=blue]{hyperref}

\title{Universal Cycles for Weak Orders}
\author{Victoria Horan\thanks{\texttt{vhoran@asu.edu}} and Glenn Hurlbert\thanks{\texttt{hurlbert@asu.edu}} \\ School of Mathematics and Statistics \\ Arizona State University \\ Tempe, AZ  85287 USA}

\begin{document}

\maketitle

\begin{abstract}
	Universal cycles are generalizations of de Bruijn cycles and Gray codes that were introduced originally by Chung, Diaconis, and Graham in 1992. They have been developed by many authors since, for various combinatorial objects such as strings, subsets, permutations, partitions, vector spaces, and designs.  One generalization of universal cycles, which require almost complete overlap of consecutive words, is $s$-overlap cycles, which relax such a constraint.  In this paper we study weak orders, which are relations that are transitive and complete.  We prove the existence of universal and $s$-overlap cycles for weak orders, as well as for fixed height and/or weight weak orders, and apply the results to cycles for ordered partitions as well.
\end{abstract}

\section{Introduction}

A \textbf{weak order on $[n]$} is a relation $\preceq$ that is transitive and complete.  We write $x \equiv y$ if $x \preceq y$ and $y \preceq x$, and we write $x \prec y$ if $x \preceq y$ but $y \not \preceq x$.  A weak order on $[n]$ can be written as a permutation of $[n]$ with consecutive symbols separated by $\equiv$ or $\prec$. (See \cite{Knuth}, fasc.~2, problem 105.)  We use the notation $\mathcal{W}(n)$ to represent the set of all weak orders on $[n]$.  For example, $\mathcal{W}(3)$ contains the weak orders:
$$\begin{array}{cccc} 1 \equiv 2 \equiv 3, & 1 \equiv 2 \prec 3, & 1 \prec 2 \equiv 3, & 1 \equiv 3 \prec 2, \\ 2 \equiv 3 \prec 1, & 2 \prec 1 \equiv 3, & 3 \prec 1 \equiv 2, & 1 \prec 3 \prec 2, \\ 2 \prec 1 \prec 3, & 1 \prec 2 \prec 3, & 3 \prec 1 \prec 2, & 2 \prec 3 \prec 1, \\ 3 \prec 2 \prec 1. \end{array}$$

We are primarily interested in finding universal cycles for $\mathcal{W}(n)$ and some of its subsets.  Let $\mathcal{C}$ be a set of $k$ strings, each of length $n$.  A \textbf{universal cycle} (ucycle) $a_0 a_1 \ldots a_{k-1}$ for  $\mathcal{C}$ is a word such that each object $c \in \mathcal{C}$ appears exactly once as a subword $a_{i+1} a_{i+2} \ldots a_{i+n}$, where subscripts are taken modulo $k$.  That is, we require objects at the end of the string to wrap-around to the beginning; for example some object will appear as the subword $a_{k-1}a_ka_0a_1 \ldots a_{n-3}$.  Note in particular that every $a_{i+1}a_{i+2} \ldots a_{i+n}$ is a string in $\mathcal{C}$.  For example, a ucycle for the set of binary strings of length three is the de Bruijn cycle $$00010111.$$

Ucycles are generalizations of de Bruijn cycles and Gray codes that were originally introduced by Chung, Diaconis, and Graham in 1992 \cite{UC}.  They have been developed for various combinatorial objects, such as binary strings, subsets, restricted multisets, permutations, partitions, lattice paths, and designs, by many authors since (see \cite{Godbole2, UC, Dewar, Subsets, Isaak, Jackson, RPKnuth, PermsUC, TiedPerms}).  Note that this definition requires our subwords to be contiguous.  Others have considered variations in which subwords need not be contiguous, such as rosaries \cite{rosaries}, combs \cite{cooper}, or omnisequences \cite{Omni}.

Some combinatorial objects do not readily lend themselves to the universal cycle structure.  For example, all permutations of $[n]$ written as strings can never be listed in a universal cycle because the first $n-1$ letters will always completely determine the last letter.  This forces a set of multiple disjoint cycles instead of one long cycle.  To deal with this problem, we may wish to either use a different representation for our set of objects, or generalize the concept of a ucycle.  Returning to our permutation example, we can define an alternate representation using the first $n-1$ letters of each permutation since the last letter is completely determined by the first $n-1$.  We can generalize ucycles by using overlap cycles, first defined in \cite{Godbole}.  An \textbf{$s$-overlap cycle} (ocycle) is an ordered listing of the objects so that the last $s$ letters of one word are the first $s$ letters of its successor in the listing.  Note that an $(n-1)$-ocycle is a ucycle.

In this paper we prove that ucycles exist for all weak orders on $[n]$ for a certain representation of weak orders that we discuss below.  (Theorem \ref{Wn}).  We also show similar results for weak orders of fixed weight and/or fixed height (Theorems \ref{Wkn}, \ref{Wnh}, and \ref{Wknh}), and for those over a fixed multiset (Lemma \ref{LWkn}), where the relevant terms are defined below.  We then apply these results to obtain ucycles for various types of ordered partitions, and construct $s$-ocycles for various sets of weak orders as well (Theorem \ref{OC}, Theorem \ref{OCL}, Theorem \ref{OCWknh}, Theorem \ref{OCWnh}).

\section{Definitions}\label{DEFN}
  For each word $w \in \mathcal{W}(n)$ we can define the \textbf{height} of an element $j \in [n]$ in the word $w$ to be the number of symbols $\prec$ that precede it in the weak order.  This gives us an alternative representation for $w \in \mathcal{W}(n)$ by a word $w_1w_2 \ldots w_n$ where letter $w_j$ is the height of element $j$ in the word $w$.  The set $\mathcal{W}(3)$ listed above contains the corresponding words:
  $$\begin{array}{cccc}
  	000, & 001, & 011, & 010, \\
	100, & 101, & 110, & 021, \\
	102, & 012, & 120, & 201, \\
	210.
  \end{array}$$  We will utilize this word representation of each weak order, and so we write $w = w_1w_2 \ldots w_n$.  Note that a weak order on $[n]$ in this representation will always have length $n$.  In fact, $\{w_1, w_2, \ldots , w_n\} = \{0, 1, 2, \ldots , h\}$ for some $h \leq n-1$.  This observation leads us to define the \textbf{height} $\hbox{\textsf{ht}}(w)$ of a weak order $w \in \mathcal{W}(n)$ to be such an $h$, i.e., $\hbox{\textsf{ht}}(w) = \max\{ h \mid h \hbox{ is the height of } j \in [n] \hbox{ in } w\}$.  It is often useful to restrict our attention to the subset of $\mathcal{W}(n)$ with a specific or fixed height.  Let $0 \leq h < n$, and define $$\mathcal{W}(n,h) = \{w \in \mathcal{W}(n) \mid \hbox{\textsf{ht}}(w) = h\}.$$

We will also consider the subsets of $\mathcal{W}(n)$ that correspond to a fixed multiset or fixed weight.  Define the \textbf{multiset} of a weak order $w = w_1w_2 \ldots w_n$, or  \textsf{ms}$(w)$, to be the unordered multiset of elements $\{w_1, w_2, \ldots , w_n\}$.  For a fixed multiset $M$, we define $$\mathcal{W}_M(n) = \{ w \in \mathcal{W}(n) \mid \hbox{\textsf{ms}}(w) = M \}.$$

The \textbf{weight} of a weak order $w = w_1w_2 \ldots w_n$ is the sum of its letters.  We denote this by $$\hbox{\textsf{wt}}(w) = \sum_{i=1}^n w_i.$$  To identify the subset of $\mathcal{W}(n)$ that contains only weak orders with weight $k$, we write $\mathcal{W}_k(n)$.  When considering $w = w_1w_2 \ldots w_n \in \mathcal{W}_k(n)$, note that $w_n = k - \sum_{i=1}^{n-1} w_i$, and so the first $n-1$ letters of $w$ completely define it.  We define the \textbf{prefix} of a word $w$ to be $w^- = w_1w_2 \ldots w_{n-1}$.  When $w^-$ completely identifies $w$ as a word in $\mathcal{W}_k(n)$, this is called the \textbf{prefix representation} or \textbf{prefix notation}.  Similarly, $w^+$, or the \textbf{suffix} of $w$, is defined to be $w_2w_3 \ldots w_n$.

When representing weak orders as words, we will use various types of abbreviations.  We define an exponential notation to represent repeated elements, i.e. $$a^i = \overbrace{a \ldots a}^{i \hbox{\footnotesize{ times}}}.$$  Also, we will represent runs of consecutive elements as: $$[i,j] = i(i+1)(i+2) \ldots (j-1)j.$$

An important and frequently used fact about weak orders is given below.

\begin{fact}\label{permWO}
	If $w = w_1 w_2 \ldots w_n \in \mathcal{W}(n)$, then any permutation $w'$ of the letters of $w$ is also in $\mathcal{W}(n)$.
\end{fact}

Let $S$ be a set of words, and define $S^- = \{w^- \mid w \in S\}$ and $S^+ = \{w^+ \mid w \in S\}$.  Then using Fact \ref{permWO}, it is clear that $\mathcal{W}^+(n) = \mathcal{W}^-(n)$.

\section{Ucycle Results}

Most of our results are obtained by finding Euler tours in graphs.  An Euler tour is a closed circuit that contains each edge of the graph exactly once.  When a graph contains an Euler tour, we say that the graph is eulerian.  To characterize eulerian graphs, we show that a graph is balanced (indegree is equal to outdegree at each vertex), and weakly connected (underlying undirected graph is connected).  This is summarized in the following theorem.

\begin{thm}\label{euler}
	\emph{(\cite{West}, p. 60)}  A directed graph $G$ is eulerian if and only if it is both balanced and weakly connected.
\end{thm}

All proofs of the following results will follow the same format.  After constructing the relevant transition graph, we prove that it is eulerian by showing that it is both balanced and connected.  While balanced is usually quite simple, weakly connected is more challenging.  Weakly connected is most often illustrated by showing an undirected walk in the underlying graph from an arbitrary vertex to some specially identified minimum vertex.

We begin with a result over the complete set of weak orders on $[n]$.  Theorem \ref{Wn} is to weak orders as de Bruijn's original theorem \cite{DBS} is to words.

\begin{thm}\label{Wn}
	For all $n \in \mathbb{Z}^+$, there exists a ucycle for $\mathcal{W}(n)$.
\end{thm}
\begin{proof}
Fix $n \in \mathbb{Z}^+$ and define the graph $G(n)$ as follows:  $$V(G(n)) = \{ v = v_1 \ldots v_{n-1} \mid v = w^- \hbox{ for some } w \in \mathcal{W}(n)\}$$ $$\hbox{ and }$$ $$E(G(n)) = \{(v_1,v_2) \mid v_1 = w^- \hbox{ and } v_2 = w^+ \hbox{ for some } w \in \mathcal{W}(n)\}.$$  Note that the edge set is well-defined, since any prefix $w^-$ of some $w \in \mathcal{W}(n)$ is also the suffix of the word $w_nw_1w_2 \ldots w_{n-1} \in \mathcal{W}(n)$.  We will explicitly show a method to construct a ucycle for $\mathcal{W}(n)$ using the graph $G(n)$ for $n \geq 2$.  Note that for $n=1$, the ucycle is the single letter 0.

By the construction of $G(n)$, it is clear that an Euler tour will correspond to a ucycle.  Using Theorem \ref{euler}, we need only show that the graph is weakly connected and balanced.  The graph is clearly balanced, since any incoming edge $(w_1 \ldots w_{n-1}, w_2 \ldots w_n)$ can be paired with the outgoing edge $(w_2 \ldots w_n, w_3 \ldots w_nw_1)$ at the vertex $w_2 \ldots w_n$.  Next, we will show a path from any vertex $w_1 \ldots w_{n-1}$ to the vertex $00 \ldots 0$.  We may apply the rotation function $\rho$, which maps $w_1 w_2 \ldots w_n$ to $w_2 \ldots w_n w_1$, as many times as necessary until we arrive at some $u = u_1u_2 \ldots u_n \in \mathcal{W}(n)$ with $u_n$ equal to the height of $w$.  Then $u$ is represented in $G(n)$ as the edge $(u_1 \ldots u_{n-1}, u_2 \ldots u_n)$.  Since $u_n$ is a maximum letter in $w$, we must also have $u_1 \ldots u_{n-1}0 \in \mathcal{W}(n)$, and so we have $(u_1 \ldots u_{n-1}, u_2 \ldots u_{n-1} 0) \in E(G(n))$.  Note that $u_2 \ldots u_{n-1}0$ has more zeros than $w_1 \ldots w_{n-1}$.  Repeating this process, we add more zeros at every step, which eventually must terminate when we arrive at the vertex with $n-1$ zeros.  Thus, we have a path from $w_1 \ldots w_{n-1}$ to $00 \ldots 0$, and so $G$ is weakly connected.

\end{proof}

In many cases, restricted subsets of combinatorial objects can be very useful.  For example, in \cite{SawadaRuskey}, Ruskey, Sawada, and Williams prove the existence of ucycles over the set of binary strings of length $n$ with weights $d$ and $d-1$, which is exactly the set of prefixes for binary strings weight $d$.  Given that restrictions on a set may yield additional information or interesting problems, we consider some subsets of $\mathcal{W}(n)$.  Define $$\mathcal{W}_k^-(n) = \{w^- \mid w \in \mathcal{W}_k(n)\}.$$

\begin{thm}\label{Wkn}
	For all $n,k \in \mathbb{Z}^+$ with $k \leq {n \choose 2}$, there exists a ucycle for $\mathcal{W}_k^-(n)$.
\end{thm}

To prove this theorem, we use the same approach as in the previous proof but need a few lemmas first to simplify our method.

\begin{lemma}\label{LWkn}
	Let $n \in \mathbb{Z}^+$, and let $M$ be some fixed multiset of size $n$.  Define the set $A$ to be the set of all permutations of $M$.  Then there exists a ucycle for $A$ using the prefix representation.
\end{lemma}

It is interesting to note that Lemma \ref{LWkn} can be applied to a multiset with all elements distinct.  In this case the set $A$ is the set of all permutations of an $n$-set, which has been a well-known and difficult ucycle problem \cite{Isaak, PermsUC}.

\begin{proof}
	Construct a graph $G_M$ with $$V(G_M) = \{v = a_1a_2 \ldots a_{n-2} \mid v = a^- \hbox{ for some } a\in A^-\}$$ $$\hbox{and}$$ $$E(G_M) = \{(v_1, v_2) \mid v_1 = a^- \hbox{ and } v_2 = a^+ \hbox{ for some } a \in A^-\}.$$  Note that the edges in this graph correspond to the elements of $A^-$, so we would like to find an Euler tour in $G_M$, which will produce a ucycle as desired.
	
	For any vertex $w_1 \ldots w_{n-4} w_{n-3}w_{n-2}$ corresponding to word $w = w_1 \ldots w_n$, there is a path in the underlying undirected graph (which we will simply call an undirected path) to the vertex $w_1 \ldots w_{n-4}w_{n-2}w_{n-3}$.  This path is:
	\begin{eqnarray*}
		w_1 \ldots w_{n-4}w_{n-3} w_{n-2} & \leftarrow & w_n w_1 \ldots w_{n-4} w_{n-3} \\
		& \leftarrow & w_{n-1} w_n w_1 \ldots w_{n-4} \\
		& \rightarrow & w_n w_1 \ldots w_{n-4}w_{n-2} \\
		& \rightarrow & w_1 \ldots w_{n-4} w_{n-2}w_{n-3}.
	\end{eqnarray*}

Thus we can always find an undirected path from one vertex to another whose difference is the transposition of elements $n-1$ and $n-2$.  Since we also have paths from a vertex $v = w^{--}$ to $\rho(v) = \rho(w)^{--}$ (where $\rho$ is the previously defined rotation function), we can find an undirected path between any two vertices that differ by adjacent transpositions.  Then, since the set of all adjacent transpositions generate all permutations of a set, there exists an undirected path between any vertices with the same multiset, so the graph $G_M$ is weakly connected.

Note that the graph $G_M$ must also be balanced, since having a fixed multiset $M$ ensures that any edge $$w_1 \ldots w_{n-2} \rightarrow w_2 \ldots w_{n-1}$$ can be balanced by the edge $$w_{n-1} w_1 \ldots w_{n-3} \rightarrow w_1 \ldots w_{n-2}.$$  Thus the graph is balanced and connected, and so is eulerian by Theorem \ref{euler}.

\end{proof}

To prove Theorem \ref{Wkn}, we will construct a new transition graph, $G_k(n)$.  Define $$V(G_k(n)) = \{v = w_1 \ldots w_{n-2} \mid v = w^- \hbox{ for some } w \in \mathcal{W}_k^-(n)\}$$ $$\hbox{and}$$ $$E(G_k(n)) = \{(v_1,v_2) \mid v_1 = w^- \hbox{ and } v_2 = w^+ \hbox{ for some } w \in \mathcal{W}_k^-(n)\}.$$  As before, note that this definition of edges is well-defined, for if $w_1w_2 \ldots w_{n-1} \in \mathcal{W}_k^-(n)$, then $w_{n-1} w_1 w_2 \ldots w_{n-2} \in \mathcal{W}_k^-(n)$.

Since the vertices in $G_k(n)$ are words of length $n-2$, we identify the vertex $v$ that is the minimum word in lexicographic order as the \textbf{minimum vertex}.  It will be useful to determine exactly what word $v$ is, and so we have the following fact, together with a lemma.

\begin{fact}\label{FWkn}
	\emph{(\cite{Knuth}, fasc.~3, p.~19)}  For every $k \in \mathbb{Z}^+$, there are unique $a,b \in \mathbb{Z}^+$ with $a > b$ so that $k = {a \choose 2} + {b \choose 1}$.
\end{fact}

\begin{lemma}\label{L2Wkn}
	Fix $n, k \in \mathbb{Z}^+$.  Define $$w = 0^{n-a-2}[0,b-1]b^2[b+1,a]$$ with $a, b \in \mathbb{Z}^+$ chosen so that $k = {a \choose 2} + {b \choose 1}$.  Then the minimum vertex $v$ in $G_k(n)$ is $v = w^{--}$.
\end{lemma}

\begin{proof}
	If we are looking for the element $w \in \mathcal{W}_k(n)$ with $w^{--}$ minimum in lex order, we may consider only the elements of $\mathcal{W}_k(n)$ with the largest two elements in positions $n-1$ and $n$.  These elements must be as large as possible in order to get the smallest elements possible in positions 1 through $n-2$, so the string we desire will have letters of the largest height possible.  Note that if $\hbox{\textsf{ht}}(w) = h$ and $w \in \mathcal{W}_k(n)$, then we must have $k \geq {h+1 \choose 2}$.  Thus we choose $w_n$ as large as possible so that ${w_n \choose 2} \leq k < {w_n+1 \choose 2} = {w_n \choose 2} + w_n$.
	
	Now, by our choice of $w_n$, the remaining weight is $0 \leq k - {w_n \choose 2} < w_n$.  Thus $k - {w_n \choose 2} \in \{0, 1, \ldots , w_n\}$, and so $k - {w_n \choose 2}$ is equal to some natural number $b \leq m$.  Then we can add another letter $b$ to the word to obtain a word whose corresponding multiset of symbols consists of at most $n$ elements.  If the cardinality of the multiset is less than $n$, we add 0's to reach cardinality $n$.  Sorting this multiset from smallest element to largest we obtain the desired weak order $w$.
\end{proof}

Before we begin the proof of Theorem \ref{Wkn}, we define one more term.  We say that there is a \textbf{duplicate at index i} if $w_i = w_{i+1} > 0$ for some $w_1 w_2 \ldots w_n \in \mathcal{W}(n)$.  Now we are finally ready to prove Theorem \ref{Wkn}.

\begin{proof}[Proof of Theorem \ref{Wkn}]
	First, note that if $k > {n \choose 2}$, then it is not possible to construct a weak order on $[n]$ with weight $k$.  The maximum weight weak order possible is one with multiset $\{0, 1, \ldots , n-1\}$, which must have weight $$\sum_{i=0}^{n-1} i = {n \choose 2}.$$  We would like to show an undirected path from any vertex $x^{--} = x_1x_2 \ldots x_{n-2}$ to the minimum vertex $v^{--} = v_1v_2 \ldots v_{n-2}$ in $G_k(n)$.  Since $x \in \mathcal{W}_k(n)^{--}$, we may define $x_{n-1}, x_n$ accordingly so that $x \in \mathcal{W}_k(n)$.  By Lemma \ref{LWkn}, we may assume that $x_1 \leq x_2 \leq \ldots \leq x_{n-2} \leq x_{n-1} \leq x_n$.  Our first observation is that if $x$ has either no duplicates or one duplicate then we must have $x = v$, and hence $x^{--} = v^{--}$.  This follows from the fact that the minimum vertex achieves weight $k$ in the most compact way possible, i.e., with either zero or one duplicate.  Now we assume that $x$ has at least two duplicates, and we have several cases depending on the relationship between $\textsf{ht}(v)$ and $\textsf{ht}(x)$.
	
	If $\textsf{ht}(x) = \textsf{ht}(v) = h$, then using Lemma \ref{LWkn} we may rewrite both weak orders as $$x = [0,h]x_{h+2}x_{h+3} \ldots x_n \hbox{ and } v = [0,h]v_{h+2}v_{h+3} \ldots v_n$$ where $x_{h+2} \leq x_{h+3} \leq \cdots \leq x_n$ and $v_{h+2} \leq v_{h+3} \leq \cdots \leq v_n$.  Now since $x \neq v$ there must be indices $i<j$ with $i,j \in \{h+2, h+3, \ldots , n\}$ so that $x_i > v_i$ and $x_j < v_j$.  Using Lemma \ref{LWkn} again, we can reorder the letters of $x^{--}$ so that we have $$x^{--} = x_ix_j[0,h]x_{h+2}x_{h+3} \ldots x_{n-2}.$$  Then we have the undirected path:
	\begin{eqnarray*}
		x_ix_j[0,h]x_{h+2}x_{h+3} \ldots x_{n-2} & \rightarrow & x_j[0,h]x_{h+2}x_{h+3} \ldots x_{n-2}x_{n-1} \hbox{ (with } x_i \hbox{ missing)}\\
		& \rightarrow & [0,h]x_{h+2}x_{h+3} \ldots x_n \hbox{ (with }x_i,x_j \hbox{ missing}) \\
		& \leftarrow & (x_j+1)[0,h]x_{h+2} x_{h+3} \ldots x_{n-1} \\
		& \leftarrow & (x_i-1)(x_j+1)[0,h]x_{h+2} x_{h+3} \ldots x_{n-2}
	\end{eqnarray*}
	Continuing in this manner, we will eventually arrive at the vertex $v^{--}$.
	
	If $\textsf{ht}(x) < \textsf{ht}(v)$, then we consider the fact that $x$ must have at least two duplicates, say at $x_i$ and $x_j$.  Using Lemma \ref{LWkn}, rewrite $x^{--}$ with $x_i, x_j$ at the front, i.e. $$x^{--} = x_i x_j x_1 x_2 \ldots x_{n-2}.$$  Then we have the undirected path:
	\begin{eqnarray*}
		x_i x_j x_1 x_2 \ldots x_{n-2} & \rightarrow & x_j x_1 x_2 \ldots x_{n-2} x_{n-1} \hbox{ (with } x_i \hbox{ missing)}\\
		& \rightarrow & x_1 x_2 \ldots x_{n-2} x_{n-1} x_n \hbox{ (with }x_i,x_j \hbox{ missing})\\
		& \leftarrow & (x_j+1) x_1 x_2 \ldots x_{n-2} x_{n-1} \\
		& \leftarrow & (x_i-1)(x_j+1)x_1x_2 \ldots x_{n-2}
	\end{eqnarray*}
	We continue to decrease $x_i$ and increase $x_j$ until either $x_i = 0$ or $x_j = \textsf{ht}(x)+1$.  In either case, we have removed a duplicate.  Continuing, we will eventually arrive at a vertex with either 0 or 1 duplicates, which as stated previously must be the minimum vertex $v^{--}$.
	
	Next, we note that it is not possible to have $\textsf{ht}(x) > \textsf{ht}(v)$, as otherwise $x$ would have been chosen as the minimum vertex ($v$ was chosen so as to have maximum height).  Thus in all cases we have constructed an undirected path from $x^{--}$ to $v^{--}$, so the graph must be weakly connected.
	
Lastly, the graph must be balanced, since any outgoing edge from vertex $w_1 \ldots w_{n-2}$ to $w_2 \ldots w_{n-1}$ can be paired with the incoming edge $w_{n-1} w_1 \ldots w_{n-3}$ to $w_1 \ldots w_{n-2}$ and this pairing gives a unique incoming edge for each outgoing edge.

\end{proof}

Next, we show that there is always a universal cycle for $\mathcal{W}(n,h)$, the set of weak orders on $[n]$ with height $h$.  We will restrict our attention to the case when $h < n-1$, for if $h = n-1$ then we are considering the set of all permutations of an $n$-set.  Ucycles for permutations can be found easily using the prefix representation (using Lemma \ref{LWkn}), or for alternative methods see \cite{Isaak, PermsUC}.  Note that the set $\mathcal{W}(n,h)$ can also be described as the set of all surjective functions from $[n]$ to $\{0, 1, \ldots , h\}$.  Ucycles for such functions are discussed and constructed in \cite{Bechel}.

\begin{thm}\label{ontoFns}
	\emph{\cite{Bechel}}  A ucycle of surjective functions from $[n]$ to $\{0, 1, \ldots , h\}$ exists if and only if $n > h+1$.
\end{thm}

We can rewrite this theorem in terms of fixed-height weak orders, as follows.

\begin{corr}\label{Wnh}
	For all $n \in \mathbb{Z}^+$ and all $h \in \mathbb{N}$ with $0 \leq h < n-1$, there exists a universal cycle for
$\mathcal{W}(n,h)$.
\end{corr}

We provide a shorter and more direct proof of Theorem \ref{ontoFns}, in terms of weak orders.

\begin{proof}
	We construct the standard transition graph $G(n,h)$ as follows.  We define
$$V(G(n,h))=\mathcal{W}^-(n,h)=\mathcal{W}^+(n,h)$$ $$\hbox{and}$$
\begin{eqnarray*}
	E(G(n,h)) & = & \{(v,w) \mid v \in \mathcal{W}^-(n,h), w \in \mathcal{W}^+(n,h), \hbox{ and } v_i = w_{i+1} \hbox{ for } 1 \leq i < n-1\}
\end{eqnarray*}
  If we think of the edge $(v,w)$ as being labeled with the word $v_1v_2 \ldots v_{n-1}w_{n-1}$, then it is clear that the set of edge labels corresponds to the set $\mathcal{W}(n,h)$.  Note that if $w_1w_2 \ldots w_n \in \mathcal{W}(n,h)$ then $w_nw_1 w_2 \ldots w_{n-1} \in \mathcal{W}(n,h)$ and so $G(n,h)$ is balanced.

To finish the proof, we must show that the graph is connected.  Define our minimum vertex in the graph to be $v^-=0^{n-h-1}[1,h]$.  Let $x^-=x_1x_2 \ldots x_{n-1}$ be an arbitrary vertex, and let $x_n$ be any symbol so that $x = x_1 x_2 \ldots x_{n-1}x_n \in \mathcal{W}(n,h)$.  We will show a path from $x^-$ to the minimum vertex $v^-$ by first illustrating that the subgraph induced by the permutations of the minimum vertex is weakly connected, and then describing a path from $x^-$ to some permutation of $v^-$.

Starting from the vertex $v^- = 0^{n-h-1}[1,h]$, we show that any sequence of adjacent transpositions applied to $v^-$ can be traversed by an undirected walk.  Since adjacent transpositions generate all permutations, this proves that all permutations of $v^-$ are connected.  Let $w_1w_2 \ldots w_{n-1}$ be a permutation of $v^-$.  We will show that the letters $w_i$ and $w_{i+1}$ may be transposed.  First, we note that $\{w_1, w_2, \ldots ,w_{n-1}\} = \{0, 1, 2, \ldots , h\}$ by the definition of $v^-$.  Then we define the desired walk as follows:
\begin{eqnarray*}
	w_1 w_2 \ldots w_{n-1} & \rightarrow & w_2 w_3 \ldots w_{n-1} w_i \\
	& \rightarrow & w_3 w_4 \ldots w_{n-1}w_i w_1 \\
	& \vdots & \hbox{ (rotations of the weak order } w_1w_2 \ldots w_{n-1} w_i) \\
	& \rightarrow & w_{i+1}w_{i+2} \ldots w_{n-1} w_i w_1 w_2 \ldots w_{i-1} \\
	& \rightarrow & w_{i+2} w_{i+3} \ldots w_{n-1} w_i w_1 w_2 \ldots w_{i-1} w_{i+1} \\
	& \rightarrow & w_{i+3} w_{i+4} \ldots w_{n-1} w_i w_1 w_2 \ldots w_{i-1} w_{i+1} w_i \\
	& \vdots & \hbox{ (rotations of the weak order } w_1w_2 \ldots w_{i-1} w_{i+1} w_i w_{i+2} w_{i+3} \ldots w_{n-1}w_i) \\
	& \rightarrow & w_1w_2 \ldots w_{i-1} w_{i+1} w_i w_{i+2} w_{i+3} \ldots w_{n-1}
\end{eqnarray*}
Note that along this path, every edge contains all letters from the set $\{w_1, w_2, \ldots , w_{n-1}\} = \{0, 1, 2, \ldots , h\}$, and so is a valid weak order from $\mathcal{W}(n,h)$.

Now let $x^- = x_1 x_2 \ldots x_{n-1} \in V(G(n,h))$ be arbitrary.  We want to define a path from $x^-$ to some permutation of the minimum vertex in $G(n,h)$ by repeatedly replacing any duplicates in $x$ with 0.  To create this path, we first define $x_n$ to be any symbol so that $x = x_1 x_2 \ldots x_n \in \mathcal{W}(n,h)$.  Then if $x$ has a duplicate at index $i$, we can replace it by following the path:
\begin{eqnarray*}
	x_1 x_2 \ldots x_{n-1} & \rightarrow & x_2 x_3 \ldots x_{n-1} x_n \\
	& \vdots & \hbox{ (rotations of } x) \\
	& \rightarrow & x_i x_{i+1} \ldots x_n x_1 x_2 \ldots x_{i-2} \\
	& \rightarrow & x_{i+1}x_{i+2} \ldots x_n x_1 x_2 \ldots x_{i-1} \\
	& \rightarrow & x_{i+2} x_{i+3} \ldots x_n x_1 x_2 \ldots x_{i-1} 0
\end{eqnarray*}
Repeating this procedure, we will eventually arrive at a vertex that is the prefix of some weak order $y$ that is a permutation of $0^{n-h}[1,h]$.  Following rotations, we will eventually arrive at a vertex that is a permutation of $v^-$.  Thus there exists a path from $x^-$ to $v^-$.

By Theorem \ref{euler}, since $G(n,h)$ is balanced and connected, it is eulerian.  Therefore we can find a
ucycle by following the Euler tour in $G(n,h)$.
\end{proof}

Finally, we prove the following result on the subset of fixed weight, fixed height weak orders on $[n]$.
\begin{thm}\label{Wknh}
	For every $n,k,h \in \mathbb{Z}^+$ with $k \leq {n \choose 2}$ and $0 \leq h < n$, there is a ucycle for $\mathcal{W}_k^-(n,h)$.
\end{thm}
\begin{proof}
	We construct the transition graph $G_k(n,h)$ as usual, with $$V(G_k(n,h)) = \mathcal{W}_k^{--}(n,h)$$ and $$E(G_k(n,h)) = \{(v,w) \mid v_{i+1} = w_i \hbox{ and } v_1v_2 \ldots v_{n-1}w_{n-1} \in \mathcal{W}_k^-(n,h)\}.$$  First, we note that the graph is even, since if $w_1 w_2 \ldots w_{n-1} \in \mathcal{W}_k^-(n,h)$, then $w_2 \ldots w_{n-1}w_1 \in \mathcal{W}_k^-(n,h)$.
	
	Next we must show that the graph is connected.  We define the minimum vertex $\textbf{v} = v_1 v_2 \ldots v_{n-2}$ by constructing a specific weak order $\textbf{w}$ in $\mathcal{W}_k(n,h)$, and then removing the last two elements.  Any weak order in $\mathcal{W}_k(n,h)$ must contain the letters $0, 1, 2, \ldots , h$, which has total weight $k' = \sum_{i=0}^h i$.  Let $\textbf{s} = s_1 s_2 \ldots s_{n-(h+1)}$ be the lexicographically minimum element of the set $\mathcal{B}_{k-k'}^{h+1}(n-(h+1))$, the set of words of length $n-(h+1)$ and weight $k-k'$ using the alphabet $\{0, 1, 2, \ldots , h\}$.  Then extend $s$ to be a weak order by defining $\textbf{w} = s_1 s_2 \ldots s_{n-(h+1)} 0 1 \ldots h$.   At this point we have $\textbf{w} \in \mathcal{W}_k^-(n,h)$, but we reorder the letters of $\textbf{w}$ so that $$\textbf{w} = [0,h]w_{h+2}w_{h+3} \cdots w_{n} \hbox{ where } w_{h+2} \leq w_{h+3} \leq \cdots \leq w_n.$$  Then define the minimum vertex to be $\textbf{v} = w_1 w_2 \cdots w_{n-2}$.
	
	Now we consider an arbitrary vertex $\textbf{x} = x_1 x_2 \cdots x_{n-2}$, and we will show that it is connected to the minimum vertex.  Since $\textbf{x} \in \mathcal{W}_k^{--}(n,h)$, we define $x_{n-1}$ and $x_n$ so that $x_1 x_2 \ldots x_n \in \mathcal{W}_k(n,h)$.  Then, by Lemma \ref{LWkn}, we know that $\textbf{x}$ must be connected to some vertex $$\textbf{y} = [0,h]y_{h+2}y_{h+3} \cdots y_{n-2} \hbox{ where } y_{h+2} \leq y_{h+3} \leq \cdots \leq y_{n-2}.$$  If $\textbf{y} = \textbf{v}$, then we are done.  Otherwise, there exists some $i,j \in \{h+2, h+3, \ldots , n-2\}$ so that $y_i > v_i$ and $y_j < v_j$.  Using rotations and Lemma \ref{LWkn}, there exists a path in $G_k(n,h)$ from $\mathbf{y}$ to the vertex $$y_iy_jy_1 y_2 \cdots y_{i-1}y_{i+1} \cdots y_{j-1} y_{j+1} \cdots y_{n-2}.$$  Define $y_{n-1}$ and $y_n$ arbitrarily so that $y_1 y_2 \cdots y_n \in \mathcal{W}_k(n,h)$.  Then we construct the following path in $G_k(n,h)$:
	\begin{eqnarray*}
		& & y_iy_jy_1 y_2 \cdots y_{i-1}y_{i+1} \cdots y_{j-1} y_{j+1} \cdots y_{n-2} \\
		& \rightarrow & y_jy_1y_2 \cdots y_{i-1}y_{i+1} \cdots y_{j-1}y_{j+1} \cdots y_{n-2} y_{n-1} \\
		& \rightarrow & y_1 y_2 \cdots y_{i-1}y_{i+1} \cdots y_{j-1} y_{j+1} \cdots y_{n-1}y_n \\
		& \rightarrow & y_2 \cdots y_{i-1}y_{i+1} \cdots y_{j-1}y_{j+1} \cdots y_n v_i \\
		& \rightarrow & y_3 \cdots y_{i-1}y_{i+1} \cdots y_{j-1}y_{j+1} \cdots y_n v_i v_j
	\end{eqnarray*}
	Reordering (by Lemma \ref{LWkn}) and shuffling/replacing elements, we can find a path to the vertex $$y_1y_2 \cdots y_{i-1}v_iy_{i+1} \cdots y_{j-1} v_j y_{j+1} \cdots y_{n-2}.$$  Note that by requiring that both $\textbf{v}$ and $\textbf{y}$ start with $[0,h]$, we ensure that these intermediate vertices contain all symbols in the set $\{0, 1, 2, \ldots , h\}$, and hence all edges represent valid weak orders in $\mathcal{W}^-_k(n,h)$.  Continuing this process, eventually we will arrive at $\textbf{v}$.  Since the graph is even and connected, it must be eulerian by Theorem \ref{euler}.
\end{proof}

\section{Overlap Cycle Results}

Many of our ucycle results for weak orders have corresponding ocycle results.  Ocycles were first introduced in \cite{Godbole} as a relaxation of ucycles when the maximum overlap size of $n-1$ might not be possible.  In this case, ocycle results may be used to discover the largest allowable overlap size with the hope of making it as large as possible.  We begin with an ocycle result that corresponds to the ucycle result given by Theorem \ref{Wn}.

\begin{thm}\label{OC}
	For all $n \in \mathbb{Z}^+$ and for all $s \in \mathbb{Z}^+$ with $1 \leq s \leq n-1$, there is an $s$-ocycle for $\mathcal{W}(n)$.
\end{thm}
To prove this, we have two cases:  $1 \leq s \leq \frac{n}{2}$ (vertices do not overlap to make a weak order), and $\frac{n}{2} < s \leq n-1$ (vertices must overlap to make a weak order).  First, we define the $s$-prefix, $w^{s-}$, and $s$-suffix, $w^{s+}$, of a word $w=w_1w_2 \cdots w_n$ as: $$w^{s-} = w_1w_2 \ldots w_s$$ $$\hbox{ and }$$ $$w^{s+} = w_{n-s+1}w_{n-s+2} \ldots w_n.$$

In light of Fact \ref{permWO}, we note that $\{w^{s-} \mid w \in \mathcal{W}(n)\} = \{w^{s+} \mid w \in \mathcal{W}(n)\}$.  In fact, these two sets are also the set $$\mathcal{W}^s(n) = \{w = w_1 w_2 \ldots w_s \mid w \hbox{ is a subword of some } w' \in \mathcal{W}(n)\}.$$

\begin{proof}[Proof of Theorem \ref{OC}]
	We define a transition graph $G^s(n)$ as follows.  Let $V(G^s(n))$ be the set of all possible overlaps, i.e. $V(G^s(n)) = \mathcal{W}^s(n)$.  Define $E(G^s(n))$ to contain one edge for each weak order by creating a directed edge $(u,v)$ for each weak order $w$ that begins with $u$ and ends with $v$.  We will show that this graph contains an Euler tour, which will give us an $s$-ocycle for $\mathcal{W}(n)$.
	
	First, $G^s(n)$ must have $d^+(u)=d^-(u)$ for all $u\in V(G^s(n))$, since any permutation of a weak order is again a weak order.  That is, if $u = u_1 u_2 \ldots u_s$ is a prefix of some weak order $u_1 u_2 \ldots u_n$, then we have incoming edge $u_{s+1} u_{s+2} \ldots u_n u_1 u_2 \ldots u_s$ and outgoing edge $u_1 u_2 \ldots u_s u_{s+1} \ldots u_n$.  Thus we can pair together each incoming edge with an outgoing edge, so we must have $d^+(u) = d^-(u)$.
	
	Lastly we must show that $G^s(n)$ is connected.  We will show that any vertex $v$ is connected to the vertex $0^s$.  Let $v = v_1 v_2 \ldots v_s \in V(G^s(n))$ with $h = \hbox{\textsf{ht}}(v)$.  If $s \leq \frac{n}{2}$, then we must have an edge $(v,u)$ to some vertex $u$ with $\hbox{\textsf{ht}}(u) < h$.  If $s > \frac{n}{2}$, the weak order $v_1 v_2 \ldots v_s v_{s+1} \ldots v_n$ is represented by the edge $(v, v_{n-s} v_{n-s+1} \ldots v_n).$  Note that \textsf{ht}$(v_{n-s} v_{n-s+1} \ldots v_n) \leq h$, and that any letter $v_i$ of maximum height in $v_{n-s}v_{n-s+1} \ldots v_n$ must have $n-s \leq i \leq s$.  Thus by repeating this procedure (at most $n$ times if $s=n-1$ and $v_{n-1}$ has maximum height in $v$), we reach some vertex $u \in V(G^s(n))$ with \textsf{ht}$(u) < h$.
	
	In either case, we have moved to a vertex with smaller height.  By repeating this procedure, we will eventually reach a vertex with height 0.  The only vertex with height 0 is the vertex $0^s$, and so we have arrived at our destination.
	
	Since our graph $G^s(n)$ is connected and $d^+(v)=d^-(v)$ for all $v \in V(G^s(n))$, we must have an Euler tour in the graph.  This Euler tour will translate to an $s$-ocycle on $\mathcal{W}(n)$.
\end{proof}

When considering $s$-ocycles for fixed weight weak orders on $[n]$, we notice that the following theorems follow immediately from their corresponding results about ucycles with very small adjustments.  Note, however, that we must consider whether or not $s$ and $n$ are relatively prime.  For example, if we consider all permutations of the set $\{1,2,3,4\}$, the transition graph for 2-ocycles is disconnected, as shown in Figure 1.

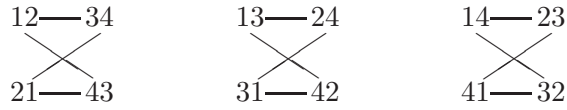
\begin{figure}
	\begin{center}
	\setlength{\unitlength}{10mm}
	\begin{picture}(8,1)
		\put(0,1){12}
		\put(0,0){21}
		\put(1,1){34}
		\put(1,0){43}
		\put(0.2,0.9){\line(3,-2){0.9}}
		\put(1.2,0.9){\line(-3,-2){0.9}}
		\qbezier(0.4,1.1)(1.4,1.1)(0.5,1.1)
		\qbezier(0.4,0.1)(1.4,0.1)(0.5,0.1)
		\put(3,1){13}
		\put(3,0){31}
		\put(4,1){24}
		\put(4,0){42}
		\put(3.2,0.9){\line(3,-2){0.9}}
		\put(4.2,0.9){\line(-3,-2){0.9}}
		\qbezier(3.4,1.1)(4.4,1.1)(3.5,1.1)
		\qbezier(3.4,0.1)(4.4,0.1)(3.5,0.1)
		\put(6,1){14}
		\put(6,0){41}
		\put(7,1){23}
		\put(7,0){32}
		\put(6.2,0.9){\line(3,-2){0.9}}
		\put(7.2,0.9){\line(-3,-2){0.9}}
		\qbezier(6.4,1.1)(7.4,1.1)(6.5,1.1)
		\qbezier(6.4,0.1)(7.4,0.1)(6.5,0.1)
	\end{picture}
	\end{center}
	\caption{Transition Graph: 2-ocycles for permutations of $\{1,2,3,4\}$}
\end{figure}

\begin{lemma}\label{OCL}
	Let $n, s \in \mathbb{Z}^+$ with $n \geq 2$ and $1 \leq s \leq n-2$ and let $M$ be some fixed multiset of size $n$.  Define the set $A$ to be the set of all permutations of $M$.  If gcd$(s,n) = 1$, then there is an $s$-ocycle for $A$.
\end{lemma}
\begin{proof}
	Small adjustments to the proof of Lemma \ref{LWkn}.
\end{proof}

Note that as stated immediately following Lemma \ref{LWkn}, Lemma \ref{OCL} can be used on a multiset with distinct elements (and hence $A$ is a set of permutations of an $n$-set).  For further extensions on this, see \cite{ME}.

\begin{thm}
	Let $n,s,k \in \mathbb{Z}^+$ with $1 \leq s \leq n-2$ and $k \leq {n \choose 2}$.  If gcd$(s,n) = 1$, then there is an $s$-ocycle for $\mathcal{W}_k(n)$.
\end{thm}
\begin{proof}
	Small adjustments to the proof of Theorem \ref{Wkn}.
\end{proof}
\begin{thm}\label{OCWknh}
	Let $n,s,h,k \in \mathbb{Z}^+$ with $1 \leq s \leq n-2$, $k \leq {n \choose 2}$, and $0 \leq h < n$.  If gcd$(s,n) = 1$, then there is an $s$-ocycle for $\mathcal{W}_k(n,h)$.
\end{thm}
\begin{proof}
	Small adjustments to the proof of Theorem \ref{Wknh}
\end{proof}

To produce an ocycle equivalent of Corollary \ref{Wnh}, we can simplify the proof since we are now dealing with overlaps of at most $n-2$ symbols; however we require that gcd$(s,n)=1$.
\begin{thm}\label{OCWnh}
	For all $n,s,h \in \mathbb{Z}^+$ with $1 \leq s \leq n-2$, gcd$(s,n)=1$, and $0 \leq h \leq n-1$, there is an $s$-ocycle for $\mathcal{W}(n,h)$.
\end{thm}
\begin{proof}
	We define our transition graph as usual, with $$V(G^s(n,h)) = \mathcal{W}^s(n,h)$$ $$\hbox{ and }$$ $$E(G^s(n,h)) = \{(v,w) \mid v_{i+1} = w_i \hbox{ and } v \in \mathcal{W}^{s-}(n,h), w \in \mathcal{W}^{s+}(n,h)\}.$$  In this transition graph, we allow multiple edges if multiple weak orders begin with the same prefix and end with the same suffix.  Clearly the graph is balanced, so we need only show that it is connected.  Define the minimum vertex $v^s$ to be the first $s$ letters of the weak order $v=[0,h]0^{n-h-1}$.  Let $x^s = x_1 x_2 \ldots x_s$ be an arbitrary vertex in the graph.  We assume that $x^s$ is an $s$-prefix of some $x = x_1x_2 \ldots x_n \in \mathcal{W}(n,h)$.  Applying Lemma \ref{OCL}, we may assume that $x$ is ordered so that we have $x = [0,h]x_{h+2}x_{h+3} \ldots x_n$.  If $s \leq h+1$, then the first $s$ letters of $x$ equals $v^s$ and we are done.  Otherwise, $s \geq h+2$, and we follow the edge that corresponds to the weak order $$[0,h]x_{h+2} x_{h+3} \ldots x_s 0 \cdots 0.$$  Applying Lemma \ref{OCL} again, we can reorder this weak order to find a path to the vertex that consists of the first $s$ letters of $[0,h]0 \cdots 0 x_{h+2} x_{h+3} \ldots x_s$, and we are one step closer to the minimum vertex.  Repeating, we will eventually arrive at the minimum vertex $v^s$.  Thus the graph is connected, and so contains an Euler tour, and hence an $s$-ocycle exists.
\end{proof}

\section{Connections to Other Combinatorial Objects}

Weak orders are equivalent to various other combinatorial objects.  For example, they are equivalent to ordered partitions, by the following theorem.

\begin{thm}\label{bij}
	\emph{(\cite{RPKnuth}, Problem 482)}  There is a bijection between the set of ordered partitions of $[n]$ with the set of weak orders on $[n]$.
\end{thm}

For example, the set $\mathcal{W}(3)$ can represent the ordered partitions (corresponding to our previous listing of $\mathcal{W}(3)$ in Section \ref{DEFN}):
$$\begin{array}{cccc}
	123 & 12|3 & 1|23 & 13|2 \\
	23|1 & 2|13 & 3|12 & 1|3|2 \\
	2|1|3 & 1|2|3 & 3|1|2 & 2|3|1 \\
	3|2|1
\end{array}$$
Using this theorem, we can easily obtain the corollaries to Theorems \ref{Wn}, \ref{Wnh}, \ref{OCWnh}.  While fixed-height weak orders correspond to ordered partitions into a fixed number of parts, it is unclear how fixed-weight weak orders to ordered partitions.

Weak orders on $[n]$ can also be described as permutations with ties.  A similar concept with subtle differences is that of tied permutations, discussed in \cite{TiedPerms}.   Using Leitner and Godbole's definition, words correspond to tournament rankings.  For example, if there is a tie for first place, then no one can win second place.  Under this definition, 113 is an allowable ranking on $[3]$ but 112 is not.  However, we note that 001 is a valid weak order on $[3]$, but 002 is not.  These differences produce distinct sets of strings, which, while order isomorphic, have different properties that must be preserved when creating ucycles and ocycles.  One can think of tied permutations and weak orders as distinct representations of ordered partitions.

\end{document}